\documentclass[12pt]{amsart}
\usepackage{amssymb,verbatim,enumerate,ifthen}
\usepackage{esint}
\usepackage[mathscr]{eucal}
\usepackage[utf8]{inputenc}
\usepackage[T1]{fontenc}
\makeatletter
\@namedef{subjclassname@2010}{%
  \textup{2020} Mathematics Subject Classification}
\makeatother

\newtheorem{thm}{Theorem}
\newtheorem{cor}{Corollary}
\newtheorem{prop}{Proposition}
\newtheorem{lem}{Lemma}

\newcommand\R{\mathbb{R}}
\newcommand\N{\mathbb{N}}

\newcommand\MM{\mathbf{M}}
\newcommand{\abs}[1]{\left| #1 \right| }

\author[P. Pasteczka]{Pawe\l{} Pasteczka}
\address{Institute of Mathematics \\ Pedagogical University of Krak\'ow \\ Podchor\k{a}\.zych str. 2, 30-084 Krak\'ow, Poland}
\email{pawel.pasteczka@up.krakow.pl}

    

\newcommand{\operator}[1]{\mathop{\vphantom{\sum}\mathchoice
{\vcenter{\hbox{\LARGE $#1$}}}
{\vcenter{\hbox{\Large $#1$}}}{#1}{#1}}\displaylimits}

\def\Mst_#1^#2{\operator{\mathscr{M}_{\mbox{\scriptsize$\#$}}\!\!}_{#1}^{#2}\,\,}

\def\eq#1{{\rm(\ref{#1})}}
\def\Eq#1#2{\ifthenelse{\equal{#1}{*}}
  {\begin{equation*}\begin{aligned}[]#2\end{aligned}\end{equation*}}
  {\begin{equation}\begin{aligned}[]\label{#1}#2\end{aligned}\end{equation}}}

\title{There is at most one continuous invariant mean}
\begin{document}

\begin{abstract}
We show that, for a (not necessarily continuous) weakly contractive mean-type mapping $\mathbf{M} \colon I^p\to I^p$ (where $I$ is an interval and $p \in \mathbb{N}$), the functional equation 
$K \circ \mathbf{M}=K$ has at most one solution in the family of continuous means $K \colon I^p \to I$. 

 Some general approach to the latter equation is also given.
\end{abstract}

\maketitle
\section{Introduction}

Celebrated result by Borwein-Borwein \cite[Theorem~8.8]{BorBor87} states that for every continuous, contractive mean-type mapping $\MM \colon I^p \to I^p$ (here and below $I$ is an  arbitrary subinterval of reals) there exists exactly one \mbox{$\MM$-invariant} mean. Moreover this mean is also continuous. 

In few recent studies by Pasteczka \cite{Pas19a} and Matkowski-Pasteczka \cite{MatPas20a,MatPas21} it was proved that if $\MM$ is not continuous then there also exist an \mbox{$\MM$-invariant} mean, but it is no longer uniquely determined. On the other hand, it was proved that in a narrow case $p=2$ every contractive mean-type mapping $\MM \colon I^2\to I$ has at most one continuous $\MM$-invariant mean (see \cite{Pas19a}). We generalize this result to all $p \ge 2$ and relax contractivity assumption.

\subsection*{Invariance principle} 
Recall that a function $M:I^{p}\rightarrow I$ is called a mean on $I$ if it is internal,
that is if%
\begin{equation*}
\min \left( v_{1},...,v_{p}\right) \leq M\left( v_{1},...,v_{p}\right) \leq
\max \left( v_{1},...,v_{p}\right) \text{, \ \ \ \ \ }v_{1},...,v_{p}\in I\,%
\text{, }
\end{equation*}%
or, briefly, if%
\begin{equation*}
\min v\leq M\left( v\right) \leq \max v\text{, \ \ \ \ \ }v\in I^{p}\text{.}%
\,\text{ }
\end{equation*}

In the sequel, to avoid the trivial results, we assume that $p>1$.

A mapping $\mathbf{M}\colon I^{p}\rightarrow I^{p}$ is referred to as 
\emph{mean-type} if there exists some means $M_{i}\colon I^{p}\rightarrow I$%
, $i=1,\dots ,p$, such that $\mathbf{M}=(M_{1},\dots ,M_{p})$.

We say that a function $F\colon I^{p}\rightarrow \mathbb{R}$ is invariant
with respect to $\mathbf{M}$ (briefly $\mathbf{M}$-invariant), if $F\circ 
\mathbf{M}=F$. Furthermore, a mean-type mapping $\mathbf{M}\colon I^{p}\rightarrow I^{p}$ is called
\emph{contractive} provided 
\begin{equation*}
\max (\mathbf{M}(v))-\min (\mathbf{M}(v))<\max (v)-\min (v)
\end{equation*}
for every nonconstant vector $v\in I^{p}$.

\section{Results}

Let us begin with the purely topological auxiliary result which turns out to be the key tool in the proof of main theorem. Remarkably, it does not contain means in its wording.
\begin{lem}\label{lem:const}
 Let $D$ be a Hausdorff, $\sigma$-compact topological space and $F \colon D \to [0,\infty)$ be a continuous function. 
 
 If $T \colon [0,\infty) \to 2^D$ is nondecreasing, right-continuous (we consider topological limit on $2^D$) and such that 
 \begin{itemize}
\item each $T(x)$ is closed;
 \item $\vec F \circ T \colon [0,\infty) \to 2^{[0,\infty)}$ is left-continuous.
 \end{itemize}

 Then $\vec F \circ T$ is constant.
 \end{lem}

\begin{proof}
Denote briefly $G:=\vec F \circ T$. Since $T$ has a limit $T(\infty):=\bigcup_{x \in [0,\infty)} T(x)$, one can extend the domain of to $G$ the set $[0,\infty]$.
As $T$ is nondecreasing (in the inclusion ordering), so is $G$. Thus it suffices to show that $G(\infty) \subseteq G(0)$. To this end, take $y \in G(\infty)$ arbitrarily. 
Define
\Eq{*}{
S_y:= \{ s \in [0,+\infty) \colon T(s) \cap F^{-1}(y) \ne \emptyset \}.
}

First let us rewrite definition of $S_y$ in two equivalent forms
\Eq{Sy_a}{
S_y= \{ s \in [0,+\infty) \colon y \in \vec F \circ T(s) \}
= \{ s \in [0,+\infty) \colon y \in G(s) \}.
}
Now it is sufficient to show that $0 \in S_y$.  First observe that for all $s_0 \in S_y$ and $s>s_0$ we have $T(s_0) \subseteq T(s)$ and thus $s\in S_y$.

Second, since $y \in G(\infty)$, there exists $u \in D$ and $a \in [0,+\infty)$ such that 
$y=F(u)$ and $u \in T(a)$. In particular $u \in T(a) \cap F^{-1}(y)$ and, consequently, $a \in S_y$. In particular $S_y$ is a nonempty interval with the right endpoint equal to infinity. We have three cases:

\medskip 

\noindent {\sc Case 1.} $S_y=(a,+\infty)$ for some $a \in [0,+\infty)$. 

Then one can take a sequence $(x_n)_{n=1}^\infty$ of elements in $D$ such that 
\Eq{*}{
x_n \in T(a+\tfrac1n) \cap F^{-1}(y) \qquad \text{ for all }n \in \N.
}
Then all $x_n$-s belong to the closed space $D_0:=T(a+1) \cap F^{-1}(y)$. However, as $D$ is $\sigma$-compact, we obtain that $D_0$ is compact. 
Thus the sequence $(x_n)_{n=1}^\infty$ contains a subsequence $(x_{n_k})_{k=1}^\infty$ which is convergent to $\tilde x \in D_0$.

Obviously $\tilde x$ belongs to the topological limit of $\big(T(a+\tfrac1n) \cap F^{-1}(y)\big)_{n=1}^\infty$, that is $\tilde x \in T(a) \cap F^{-1}(y)$. This implies that $a \in S_y$, which is the contradiction.

\medskip
\noindent {\sc Case 2.} $S_y=[a,+\infty)$ for some $a \in (0,+\infty)$. In this case, by the one-sided continuity, we have 
\Eq{*}{
y \in \vec F \circ T(a) = \bigcup_{x \in [0,a)} \vec F \circ T(x). 
}
Thus there exists $x_0 \in [0,a)$ such that $y \in \vec F \circ T(x_0)$. Whence $x_0 \in S_y$ contradicting our assumption.

\medskip
\noindent {\sc Case 3.} $S_y=[0,+\infty)$. In this final case we have $0 \in S_y$, which by \eq{Sy_a} implies $y \in G(0)$.

\medskip 

Finally, as $y$ is an arbitrary element of $G(\infty)$ we get the inclusion $G(\infty)\subseteq G(0)$. Consequently the equality $G(0)=G(\infty)$ is valid, which completes the proof.
\end{proof}

Now we can proceed to the main result of this paper, which was announced in the abstract.

\begin{thm}\label{thm:main}
Let $p \in \N$ and $\MM \colon I^p \to I^p$ be a contractive mean-type mapping. Then there exists at most one continuous $\MM$-invariant mean.
\end{thm}

\begin{proof}
 Let $K_1,K_2 \colon I^p \to I$ be continuous $\MM$-invariant means. We show that $K_1=K_2$.
 
 Define $F \colon I^p \to [0,+\infty)$ and $T \colon [0,+\infty) \to 2^{I^p}$ by 
 \Eq{*}{
 F(v)&:=\abs{K_1(v)-K_2(v)}, \\
 T(a)&:=\{v \in I^p \colon \max(v)-\min(v) \le a \}.
 }
 
 Then $T$ in nondecreasing and right-continuous, that is 
 \Eq{*}{
 \bigcap_{\alpha>a} T(\alpha)=T(a)\text{ for all }a>0.
 }
 Furthermore each $T(a)$ is a closed subset of $I^p$.
 
 Now we show that $\vec F \circ T$ is left-continuous. To this end, take $a>0$ arbitrarily. Obviously, since $T$ is monotone, one gets 
 \Eq{*}{
 \vec F \circ T(a^-) := \vec F \Big( \bigcup_{\alpha\in[0,a)} T_\alpha \Big) \subseteq \vec F \circ T(a).
 }
 For the converse implication take any element $y \in \vec F \circ T(a)$. Then there exists a nonconstant vector $x \in T(a) \subseteq I^p$ such that $y=F(x)$. Then, since $F$ is $\MM$-invariant, we also have $y=F \circ \MM(x)$. However $\MM$ is contractive, thus
 \Eq{*}{
 \max (\MM(x))-\min (\MM(x)) < \max(x)-\min(x) \le a. 
 }
 Consequently $\MM(x) \in T(a^-)$, and $y=F \circ \MM(x) \in \vec F \circ T(a^-)$.
 
 This yields $\vec F \circ T(a)=\vec F \circ T(a^-)$ and $\vec F \circ T$ is left-continuous.
 
According to Lemma~\ref{lem:const} function $\vec F \circ T$ is constant. In particular
\Eq{*}{
\vec F (I^n) = \bigcup_{a>0} \vec F \circ T(a)=\vec F \circ T(0)=\vec F \big(\big\{(x,\dots,x)\colon x\in I\big\}\big)=\{0\}.
}
Thus $F \equiv 0$. Therefore $K_1(v)=K_2(v)$ for all $v \in I^p$.
\end{proof}

\section{Applications}
\subsection{Weakly contractive mean-type mappings}
This section extends some consideration contained in \cite{MatPas21}. 
We say that a mean-type mapping $\mathbf{M}\colon I^{p}\rightarrow I^{p}$ is 
\emph{weakly contractive} if for every nonconstant vector $v\in I^{p}$ there
is a positive integer $n_0\left( v\right) $ such that 
\begin{equation*}
\max (\mathbf{M}^{n}(v))-\min (\mathbf{M}^{n}(v))<\max (v)-\min (v) \qquad \text{ for all }n \ge n_0(v).
\end{equation*}

Let us emphasize that it is sufficient to verify if the inequality above is valid for $n=n_0(v)$. 
Moreover in a special case $p=2$ it was proved \cite{MatPas20a} that $%
\mathbf{M}$ is weakly contractive if and only if $\mathbf{M}^{2}$ is
contractive. However, due to \cite{MatPas21}, for every $p>2$ we can construct
weakly contractive mean-type mapping on $I^p$ such that the function $I^p \ni v \mapsto n_0(v)$ is unbounded. 

Now recall a sufficient condition to guarantee the uniqueness of invariant mean.
\begin{prop}[\cite{MatPas21}, Theorem 2]
If $\mathbf{M}:I^p\to I^p$ is a continuous, weakly contractive mean-type mapping then there exists a unique $\mathbf{M}$-invariant mean $K\colon I^p\rightarrow I$. Moreover, the sequence of iterates $\left(\mathbf{M}^n\right) _{n\in \mathbb{N}}$ converges (pointwise on $I^{p})$ to $\mathbf{K}:=\left( K,\dots,K\right)$.
\end{prop}

Now we extend Theorem~\ref{thm:main} to the family of all weakly contractive mean-type mappings.

\begin{cor}\label{cor:main}
Let $p \in \N$ and $\MM \colon I^p \to I^p$ be a weakly contractive mean-type mapping. Then there exists at most one continuous $\MM$-invariant mean.
\end{cor}

\begin{proof}
 Since $\MM$ is weakly contractive, the mapping $\MM^* \colon I^p \to I^p$ given by $\MM^*(v):=\MM^{n_0(v)}(v)$ is contractive.
 Furthermore for every $\MM$-invariant mean $K \colon I^p \to I$, $v \in I^p$ and $n \in \N$ we have
$K \circ \MM^n(v)=K(v)$. For $n:=n_0(v)$, this equality simplifies to $K \circ \MM^*(v)=K(v)$. Thus every $\MM$-invariant mean is also $\MM^*$-invariant.

However, since $\MM^*$ is contractive, Theorem~\ref{thm:main} implies that there exists at most one continuous, $\MM^*$-invariant mean. Consequently there exists at most one continuous $\MM$-invariant mean.
\end{proof}

\subsection{An applications in solving a functional equation}
In this section, similarly to \cite{MatPas20a}, we show a simple application of our results in solving functional equations
\begin{thm}
Let $\mathbf{M}\colon I^{p}\rightarrow I^{p}$ be a weakly-contractive mean-type mapping
such that there exists a continuous $\MM$-invariant mean $K:I^{p}\rightarrow I$. 

A continuous function $F:I^{p}\rightarrow \mathbb{R}$ is $\MM$-invariant 
if, and only if, there is a continuous function $\varphi\colon I \to \R$ such that $F=\varphi \circ K$.
\end{thm}
\begin{proof}
 The $(\Leftarrow)$ implication is trivial. To show the converse assume that $F \colon I^p \to \R$ satisfies $F \circ \MM =F$ and $K$ is the continuous \mbox{$\MM$-invariant} mean.
Define  $\varphi\colon I \to \R$ by
\Eq{*}{
\varphi(x):=F(x,\dots,x) \qquad (x\in I).
}
Assume to the contrary that $F(v) \ne \varphi \circ K (v)$ for some $v \in I^p$. Then there exists $\varepsilon\in(0,+\infty)$ such that the set
\Eq{*}{
E:=\{v\in I^p \colon \abs{\varphi \circ K(v)-F(v)}\ge \varepsilon\}
}
is nonempty. Since both $K$ and $F$ are $\MM$-invariant, we obtain
\Eq{*}{
\varepsilon \le \abs{\varphi \circ K(v)-F(v)}=\abs{\varphi \circ K\circ \MM(v)-F \circ \MM(v)} \text{ for all }v\in E,
}
thus $\MM(E) \subseteq E$.

Next, since $E$ is closed, there exists $v_0 \in E$ such that
\Eq{v_0min}{
\max(v_0)-\min(v_0)=\inf_{v \in E} \max(v)-\min(v).
}

As $E$ is does not contain constant vectors one gets 
\Eq{*}{
\max(v_0)-\min(v_0)>0.
}
As $\MM$ is weakly-contractive, let $v_1:=\MM^{n(v_0)}(v_0)$. Then
\Eq{*}{
&\quad \max(v_1)-\min(v_1)<\max(v_0)-\min(v_0)\\
\text{ and }&\quad v_1=\MM^{n(v_0)}(v_0) \in \MM^{n(v_0)}(E) \subseteq E,
}
contradicting \eq{v_0min}.

Therefore $F(v)=\varphi \circ K (v)$ for all $v \in I^p$.
\end{proof}

\end{document}